\documentclass{article}
\usepackage{ amsmath, amsfonts, amssymb,amsthm}
\usepackage{color, graphicx}

\newtheorem{definition}{Definition}[section]
\newtheorem{lemma}[definition]{Lemma}
\newtheorem{remark}[definition]{Remark}

\newtheorem{proposition}[definition]{Proposition}

\newtheorem{theorem}[definition]{Theorem}

\def\reg{\operatorname{reg}}
\begin{document}

\makeatletter      
\renewcommand{\ps@plain}{%
     \renewcommand{\@oddhead}{\textrm{ON THE REGULARITY INDEX OF $s$ EQUIMULTIPLE FAT POINTS}\hfil\textrm{\thepage}}%
     \renewcommand{\@evenhead}{\@oddhead}%
     \renewcommand{\@oddfoot}{}
     \renewcommand{\@evenfoot}{\@oddfoot}}
\makeatother     

\title{ON THE REGULARITY INDEX OF $s$ EQUIMULTIPLE FAT POINTS NOT ON A LINEAR $(r-1)$-SPACE, $s \le r+3$}         
\author{Phan Van Thien,  Ho Thi Doan Trang}        
\date{}          
\maketitle

\pagestyle{plain}

\begin{abstract}\noindent We prove the Trung's conjecture about Segre's upper
bound for $s$ equimultiple fat points not on a linear $(r-1)$-space, $s\le r+3$,
by algebraic method used in \cite{CTV}. This method also may used to research
other cases of fat points.
\end{abstract}



\par \section{Introduction} \ \ \ \ Let $K$ be an an
algebraically closed of arbitrary characteristic and let $P_1, \ldots, P_s$ be
distinct points in the  projective space $\mathbb P^n:=\mathbb P^n_K$. Denote by
$\wp_1, \ldots, \wp_s$ the homogeneous prime ideals of the polynomial ring
$R:=K[X_0, \ldots, X_n]$ corresponding to the points $P_1, \ldots, P_s$.\par

Let $m_1, \ldots, m_s$ be positive integers. Denote by $m_1P_1+\cdots+m_sP_s$
the zero-scheme defined by the ideal $I:=\wp^{m_1}_1 \cap \cdots \cap
\wp^{m_s}_s$. Then $$Z:=m_1P_1+\cdots+m_sP_s$$ is called to be a set of fat
points in $\mathbb P^n$. \par

The ring $R/I$ is the homogeneous coordinate ring of $Z$.  It is a graded ring,
$R/I=\underset{t \ge 0}{\oplus} (R/I)_t$, whose multiplicity is $e(R/I)
:=\underset{i=1}{\overset{s}{\sum}} \binom{m_i+n-1}{n}$. We also call it to be
the multiplicity of $Z$, and denote it by $\reg(Z)$.

The Hilbert function of $Z$ is defined to be $H_Z(t):=\dim_K (R/I)_t$. This
function strictly increases until it reaches $e(Z)$, at which it stabilizes.
Then the number
$$\min\{ t\in \mathbb Z | H_A(t)=e(A)\}$$
is called the regularity index of $Z$, and denote it by $\reg(Z)$. It is well
known that $\reg(Z)=\reg(R/I)$, the Castelnuovo-Mumford regularity of $R/I$.\par

The problem to exactly determine $\reg(Z)$ is more fairly difficult. So, instead
of it one to find an upper bound for $\reg(Z)$. \par

For generic fat points $Z=m_1P_1+\cdots+m_sP_s$ in $\mathbb P^2$ with $m_1 \ge
\cdots \ge m_s$, Segre \cite{S} showed that
$$\reg(Z) \le \max\left\{m_1+m_2-1, \left[\frac{m_1+\cdots+m_s}2\right]\right\}.$$

A set of fat points $Z=m_1P_1+\cdots+m_sP_s$ in $\mathbb P^n$ is said to be in
general position if no $j+2$ of the points $P_1, \ldots, P_s$ are on any
$j$-plane for $j <n$. A set of generic fat points always in general position.
Segre's upper bound later was  generalised by Catalisano, Trung, Valla
\cite{CTV} for fat points in general position in $\mathbb P^n$

$$\reg(Z) \le
\max\left\{m_1+m_2-1, \left[(\sum_{i=1}^s m_i+n-2)/n\right]\right\}.$$

In 1996, Trung conjectured an upper bound for the regularity index of arbitrary
fat points $Z=m_1P_1+\cdots+m_sP_s$ in $\mathbb P^n$ (see \cite{Th2},
\cite{Th3}).

\medskip

\noindent {\bf Conjecture.} {\it
$$\reg(Z) \le \max \{ T_j |\  j=1,\ldots,n \},$$
where
$$T_j = \max\left\{\left[\frac{\sum_{l=1}^q m_{i_l}+ j- 2}{j}\right] |\
P_{i_1}, \ldots , P_{i_q} \text{ lie on a linear }j\text{-space}\right\}.$$}

\medskip

He called this upper bound to be the Segre's upper bound. This upper bound
nowadays referred to as Segre's bound.\par

The same conjecture was also given independently by Fatabbi and Lorenzini (see
\cite{Fa1}, \cite{Fa2}).\par

The Trung's conjecture has been proved in many cases: $n=2$ (see \cite{Fa1},
\cite{Th1}), $n=3$ (see \cite{Fa2}, \cite{Th2}), $n=4$ and $m_1=\cdots=m_s=2$
(see \cite{Th3}), for $n+2$ non-degenerate fat points in $\mathbb P^n$ (see
\cite{Be}), for $s+2$ fat points whose support not on a linear $(s-1)$-space
(see \cite{Th6}), for $n+3$ non-degenerate almost equimultiple fat points in
$\mathbb P^n$ (see \cite{TH}), and recently for $n+3$ non-degenerate fat points
in $\mathbb P^n$ (see \cite{TH}).

\medskip

\par A set of $s$ points $P_1,
\ldots, P_s$ in $\mathbb P^n$ is said to be in general position on a linear
$r$-space $\alpha$  if all points $P_1, \ldots, P_s$ lie on the $\alpha$ and  no
$j+2$ of these points lie on a linear $j$-space for $j <r$. So if $r=n$, then we
get the case of the points which are in general position in $\mathbb P^n$.

\par A set of $s$ fat points $Z=m_1P_1+\cdots +m_sP_s$ in $\mathbb P^n$ is said
to be equimultiple if  $m_1=\cdots=m_s$.
\par

In this paper  by algebraic method used in \cite{CTV}, we prove the Trung's
conjecture about Segre's upper bound for $s$ equimultiple fat points not on a
linear $(r-1)$-space, $s\le r+3$. This method also may used to research other
cases of fat points.

\section{Preliminaries} \ \ \ \ From now on, we say a $j$-plane, i.e. a linear
$j$-space. We identify a hyperplane as the linear form defining it.\par

We use the following lemmas which have been proved in \cite{CTV}, \cite{Th6}.
The first lemma allows us to compute the regularity index by induction.

\begin{lemma}\label{lem21} \cite[Lemma 1]{CTV} Let $P_1,\ldots , P_r, P$ be
distinct points in $\Bbb P^n$ and let $\wp$ be the defining ideal of $P$. If
$m_1,\ldots , m_r$ and $a$ are positive integers, $J := \wp^{m_1}_1\cap
\cdots\cap  \wp^{m_r}_r$, and $I = J \cap \wp^a$, then $$\reg(R/I) =
\max\left\{a-1, \reg(R/J), \reg(R/(J+\wp^a)) \right\}.$$ \end{lemma}

To estimate  $\reg(R/(J+\wp^a))$ we shall  use the following lemma. \par
\begin{lemma}\label{lem22} \cite[Lemma 3]{CTV}
 Let $P_1,\ldots , P_r$ be distinct
points in $\Bbb P^n$ and $a, m_1,\ldots , m_r$ positive integers. Put $J =
\wp^{m_1}_1\cap \cdots\cap  \wp^{m_r}_r$ and $\wp = (X_1,\ldots , X_n)$. Then
$$\reg(R/(J+\wp^a)) \le b$$ if and only if $X^{b-i}_0M \in J+\wp^{i+1}$ for
every monomial $M$ of degree $i$ in $X_1,\ldots , X_n$, $i = 0,\ldots , a-1$.
\end{lemma}  \par

Suppose that we can find $t$ hyperplanes $L_1, \ldots, L_t$ avoiding $P$ such
that $L_1 \cdots L_t M \in J$. For $j=1, \ldots, t$, since we can write
$L_j=X_0+G_j$ for some linear form $G_j\in \wp$, we get $X_0^{t}M \in
J+\wp^{i+1}$. Therefore, we have the following remark:

\begin{remark}\label{remark23} Assume that $L_1, \ldots, L_t$ are hyperplanes
avoiding $P$ such that $L_1 \cdots L_t  M \in J$ for every monomial $M$ of
degree $i$ in $X_1,\ldots , X_n$, $i = 0,\ldots , a-1$. If
$$\delta=\max\{t+i|  0\le i \le a-1\},$$
then $$\reg(R/(J+\wp^a)) \le \delta.$$
\end{remark}

\par\smallskip

The following lemma are the main results of \cite{Th6}. \par

\begin{lemma}\label{lem28} \cite[Theorem 3.4]{Th6} Let $P_1, \ldots, P_{s+2}$ be
distinct points not on a linear $(s-1)$-space in $\mathbb P^n$, $s \le n$, and
$m_1, \ldots, m_{s+2}$ be positive integers. Put $I=\wp_1^{m_1} \cap \cdots \cap
\wp_{s+2}^{m_{s+2}}$, $A=R/I$. Then, $$\reg(A) = \max \{ T_j |\  j=1,\ldots, n
\},$$ where $$T_j = \max\left\{\left[\frac{\sum_{l=1}^q m_{i_l}+ j- 2}{j}\right]
|\ P_{i_1}, \ldots , P_{i_q} \text{ lie on a linear }j\text{-space}\right\},$$
$j=1,\ldots, n$. \end{lemma}\par

\section{Regularity index of $s$ equimultiple fat points not on a
linear $(r-1)$-space, $s\le r+3$}

\ \ \ \ The following lemma  help us  to find a sharp upper bound for the
regularity index of $s$  fat points in $\mathbb P^n$.

\begin{lemma}\label{lem41} Let  $P_1, \ldots, P_s, P$ be distinct points in
$\mathbb P^n$ such that for $r$ arbitrary points of $\{P_1, \ldots, P_s\}$,
there always exists a linear $(r-1)$-space passing through these $r$ points and
avoiding $P$. Let $m_1, \ldots, m_s$ be positive integers. Consider the set
$\{P_1, \ldots, P_s\}$ with the chain of multiplicities $(m_1, \ldots, m_s)$.
Assume that $t$ is an integer such that
$$t \ge \max\left\{m_j, \left[\frac{\sum_{i=1}^s m_{i}+ r- 1}{r}\right]|\ j=1, \ldots s \right\}.$$
Then, there exist $t$ linear $(r-1)$-spaces, say $L_1, \ldots, L_t$, avoiding
$P$ such that for every point $P_j\in \{P_1, \ldots, P_s\}$, there are $m_j$
linear $(r-1)$-spaces (including multiplicity) of $\{L_1, \ldots, L_t\}$ passing
through the $P_j$.

\end{lemma}

\begin{proof} We argue by induction on $\sum_{i=1}^s m_i$.  We may assume that
$m_1\ge \cdots \ge m_s$ (after relabelling, if necessary). If $s\le r$, then by
the assumption there is a linear $(r-1)$-space, say $L$, passing through all
points $P_1, \ldots, P_s$
 and avoiding $P$. Let $L_1=\cdots =L_t=L$. Since $t\ge \max\{m_j | j=1, \ldots, s\}$, for every point
 $P_j$ there exits $m_j$ linear $(r-1)$-spaces of $\{L_1, \ldots, L_t\}$ passing
through the $P_j$.

If $s>r$, then by the assumption there is a linear $(r-1)$-space, say $L_1$,
passing through all points $P_1, \ldots, P_r$ and avoiding $P$. Since $t\ge
\left[\frac{\sum_{i=1}^s m_i+ r- 1}{r}\right]$, we have
$$t-1\ge \left[\frac{(m_1-1)+\cdots+(m_r-1)+m_{r+1}+\cdots+m_s+ r- 1}{r}\right].$$
On the other hand, since $t\ge \left[\frac{\sum_{i=1}^s m_i+ r- 1}{r}\right]$
and $m_1\geq m_2\geq \cdots\geq m_{s}$, we have
$$t-1\ge \left[\frac{(r+1)m_{r+1}+r- 1}{r}\right]-1\ge m_{r+1}.$$
So,
$$t-1 \ge \max\left\{m_1-1, \ldots, m_r-1, m_{r+1}, \ldots, m_s
\right\}.$$

Consider the set $\{P_1, \ldots, P_s\}$ with the chain of multiplicities
$(m_1-1, \ldots, m_r-1, m_{r+1}, \ldots, m_s)$. By inductive assumption we can
find  $t-1$ linear $(r-1)$-spaces, say $L_2, \ldots, L_t$, avoiding $P$ such
that for $j=1, \ldots, r$ there are $m_j-1$ linear $(r-1)$-spaces of
$\{L_1,...,L_t\}$ passing through the point $P_j$; for $j=r+1, \ldots, s$ there
are $m_j$ linear $(r-1)$-spaces of $\{L_1,...,L_t\}$ passing through the point
$P_j$. Therefore, we have $t$ linear $(r-1)$-space $L_1,...,L_t$ as desired.

\end{proof}

\begin{lemma}\label{lem42}
Let $X=\{P_1, \ldots, P_{s+3}\}$ be a set of distinct points lie on a linear
$s$-space in $\mathbb P^n$, $3\le s \le n$,   such that there is not any linear
$(s-1)$-space containing $s+2$ points of $X$ and there is not any linear
$(s-2)$-space containing $s$ points of $X$. Let $\wp_1, \ldots, \wp_{s+3}$ be
the homogeneous prime ideals of the polynomial ring $R=K[X_0, \ldots, X_n]$
corresponding to the points $P_1, \ldots, P_s$. Assume that there is a linear
$(s-1)$-space, say $\alpha$, containing $s+1$ points $P_1, \ldots, P_{s+1}$ and
there is a linear $(s-1)$-space, say $\beta$, containing $s+1$ points $P_3,
\ldots, P_{s+3}$. Let $m$ be a positive integer. For $j=1, \ldots, n$, put
$$T_j = \max\left\{\left[\frac{mq+ j- 2}{j}\right] |\
P_{i_1}, \ldots , P_{i_q} \text{ lie on a linear }j\text{-space}\right\}.$$
Then,
$$\reg(R/(J+\wp_{s+3}^m))\le \max\{T_j|j=1,\ldots,n\},$$
where $J=\wp_1^m \cap \cdots \cap \wp_{s+2}^m.$
\end{lemma}
\begin{proof}
We remark that there exists a linear $(s-1)$-space, say $\gamma$, containing
$P_1$, $P_2$, $P_{s+2}$, $s-3$ points of $\{P_3, \ldots, P_{s+1}\}$ and avoiding
$P_{s+3}$. In fact, we assume that $\pi$ is a linear $(s-1)$-space containing
$P_1$, $P_2$, $P_{s+2}$, $s-3$ points $P_5, \ldots, P_{s+1}$ and $P_{s+3}$. Then
let $\gamma$ be the linear $(s-1)$-space containing $P_1$, $P_2$, $P_{s+2}$,
$P_4, P_5, \ldots, P_{s}$. We have $P_{s+3}\notin \gamma$ (If $P_{s+3} \in
\gamma$, then $\gamma=\pi$ is a linear $(s-1)$-space containing $s+2$ points
$P_1, P_2, P_4, \ldots, P_{s+3}$ of $X$, a contradiction).\par

We may assume that $P_1, P_2, P_4, P_5, \ldots, P_{s}, P_{s+2}\in \gamma$. Since
arbitrary $s$ points of $\beta \cap X$ do not lie on a linear $(s-2)$-space, we
can put $P_{s+3}=(1, 0, \ldots, 0)$, $P_3=(0, 0, 1, 0, \ldots, 0)$, $P_5=(0, 0,
0, \underset{4}{\underbrace{1}}, 0, \ldots, 0)$, ..., $P_j=(0, \ldots, 0,
\underset{j-1}{\underbrace{1}}, 0, \ldots, 0)$, $j=5, \ldots, s+2$. Since
$P_2\notin \beta$, we can put $P_2=(0, 1, 0, \ldots, 0)$.\par

For every monomial $M=X_1^{c_1}\cdots X_{n}^{c_n}$, $c_1+\cdots +c_n=i$, $i=0,
\ldots, m-1$. Put  $m_1=m_4=m$, $m_2=m-i+c_1$, $m_3=m-i+c_2$, $m_j=m-i+c_{j-2}$,
$j=5, \ldots, s+2$. Let $H$ be a hyperplane containing $\alpha$ and avoiding
$P_{s+3}$, let $L$ be a hyperplane containing $\gamma$ and avoiding $P_{s+3}$.
Put
$$t=\max\{m_3, m_{s+1}\}.$$
Since $c_s+\max\{c_2, c_{s-1}\} \le i$, we have $m_{s+2}+t\le 2m-i$. We consider
the following cases:

\par\smallskip\noindent {\it\bf Case 1:} $m_{s+2}+t\le 2m-i-1$, or  $s=3$, or $s=4$
and $m=2$. We have $P_1,P_2, P_4, \ldots,P_s \in H\cap L$; $P_3, P_{s+1} \in H$;
$P_{s+2}\in L$. Therefore,
$$H^{\max\{m-m_{s+2}, t\}} L^{m_{s+2}} \in \wp_1^m\cap \wp_2^m \cap \wp_3^t
\cap \wp_4^m \cap \cdots \cap \wp_{s}^m \cap \wp_{s+1}^t\cap
\wp_{s+2}^{m_{s+2}}.$$ Moreover, since $M\in
\wp_3^{m-m_3}\cap\wp_{s+1}^{m-m_{s+1}}\cap \wp_{s+2}^{m-m_{s+2}}$ and
$t=\max\{m_3, m_{s+1}\}$, we have
$$H^{\max\{m-m_{s+2}, t\}} L^{m_{s+2}} M \in \wp_1^m \cap \cdots \cap \wp_{s+2}^m=J.$$
By Remark \ref{remark23} we get
\begin{align*}\reg(R/(J+\wp^m_{s+3}))&\le \max\{\max\{m-m_{s+2}, t\}+m_{s+2}+i| i=1, \ldots, m-1\}\\
&\le \max\{m+i, t+m_{s+2}+i| i=0, \ldots, m-1 \}.\end{align*} If $m_{s+2}+t\le
2m-i-1$, then
$$\max\{m+i, t+m_{s+2}+i| i=0, \ldots, m-1 \} \le 2m-1=T_1.$$
If $s=3$, then
$$\max\{m+i, t+m_{s+2}+i| i=0, \ldots, m-1 \} \le 2m=T_2.$$
If $s=4$ and $m=2$, then
$$\max\{m+i, t+m_{s+2}+i| i=0, \ldots, m-1 \} \le 2m=4=T_4.$$

\par\smallskip\noindent {\it\bf Case 2:} $m_{s+2}+t= 2m-i$ and $s\ge 4$ and $m\ge 3$,
or $m_{s+2}+t= 2m-i$ and $s\ge 5$ and $m=2$. Without loss of generality we can
assume that $m_3 \ge m_{s+1}$. Then $t=\max\{m_3, m_{s+1}\}=m_3$. We have
$2m-i=m_{s+2}+t=m_{s+2}+m_3=m-i+c_s+m-i+c_2$. This implies $c_2+c_s=i$. So,
$c_j=0$, for every $2 \ne j \ne s$. We consider two following cases for $i$:\par

\par\smallskip\noindent {\it\bf Case 2.1:} $i=0$. By the assumption, there is not any
linear $(s-2)$-space containing $s$ points of $X$, we have $P_3$, $P_{s+1}$,
$P_{s+2}$ and $P_{s+3}$ not on a linear $2$-space. Therefore, there is a
hyperplane, say $\sigma$, containing $P_3$, $P_{s+1}$, $P_{s+2}$ and avoiding
$P_{s+3}$. Recall that  $P_1, \ldots, P_{s+1} \in H$ and $P_1, P_2, P_4, \ldots,
P_s, P_{s+2}\in L$. Thus, we have
$$H^{m-1}L^{m-1}\sigma \in \wp_1^m \cap \cdots \cap \wp_{s+2}^m=J.$$
It follows that \begin{align}H^{m-1}L^{m-1}\sigma M\in J \end{align} with
$m-1+m-1+1+i=2m-1=T_1$.

\par\smallskip\noindent {\it\bf Case 2.2:} $i\ge 1$. We consider three following cases for $c_s$:

\par\smallskip $\bullet$ $c_s < i$: Since $P_3, \ldots, P_{s+3}$ lie on the linear $(s-1)$-space $\beta$ and there
is not any linear $(s-2)$-space containing $s$ points of $X$, the linear
$(s-1)$-space containing $s$ points $P_1, P_3, \ldots, P_s, P_{s+2}$ avoids
$P_{s+3}$. Then there exists a hyperplane, say $\pi$, containing $P_1, P_3,
\ldots, P_s, P_{s+2}$ and avoiding $P_{s+3}$. Moreover, since $P_1, \ldots,
P_{s+1} \in H$ and $P_1, P_2, P_4, \ldots, P_s, P_{s+2}\in L$, we have
$$H^{m_3-1}L^{m_{s+2}-1}\pi \in \wp_1^{m_{s+2}+m_3-1} \cap \wp_2^{m_{s+2}+m_3-2}\cap \wp_3^{m_3}\cap
\wp_4^{m_{s+2}+m_3-1}\cap \cdots \cap \wp_s^{m_{s+2}+m_3-1}\cap
\wp_{s+1}^{m_3-1} \cap \wp_{s+2}^{m_{s+2}}.$$ Since $m_{s+2}+m_3=2m-i$ and $i\le
m-1$, we have $m_{s+2}+m_3-1=2m-i-1\ge m$. Since $i\ge 1$, we have $m_2=m-i \le
m-1\le m_{s+2}+m_3-2$. Since $c_2+c_s=i$ and $c_s <i$, we have $c_2 \ge 1$. Note
that $c_{s-1}=0$. Thus, $m_{s+1}=m-i+c_{s-1} \le m-i+c_2=m_3-1$. Therefore, we
have
$$H^{m_3-1}L^{m_{s+2}-1}\pi \in \wp_1^{m} \cap \wp_2^{m-i}\cap \wp_3^{m-i+c_2}\cap
\wp_4^{m}\cap \cdots \cap \wp_s^{m}\cap \wp_{s+1}^{m-i+c_{s-1}} \cap
\wp_{s+2}^{m-i+c_s}.$$ Note that $c_j=0$ for every $2\ne j\ne s$. So, $M\in
\wp_2^{i} \cap \wp_3^{i-c_2}\cap \wp_5^{i} \cap\cdots \cap \wp_{s}^{i} \cap
\wp_{s+1}^{i-c_{s-1}} \cap \wp_{s+2}^{i-c_s}$. Therefore,
\begin{align}H^{m_3-1}L^{m_{s+2}-1}\pi M \in \wp_1^m \cap \cdots \cap
\wp_{s+2}^m=J\end{align} with $m_3-1+m_{s+2}-1+1+i=2m-i-1+i=2m-1=T_1$.

\par\smallskip $\bullet$ $c_s = i$ and $m\ge 3$: Then we have $c_j=0, j=1, \ldots, s-1$. Since
$P_1, P_3, P_4, \ldots, P_{s}, P_{s+2}, P_{s+3}$ do not lie on a linear
$(s-1)$-space, there exists a hyperplane, say $\sigma_1$, containing $P_1, P_3,
P_4, \ldots, P_{s}, P_{s+2}$ and avoiding $P_{s+3}$ (If  $P_1, P_3, P_4, \ldots,
P_{s}, P_{s+2}, P_{s+3}$ lie on a linear $(s-1)$-space, then this linear
$(s-1)$-space contains $P_{s+1}$. It follows that there is a linear
$(s-1)$-space containing $s+2$ points of $X$, a contradiction).  Similarly,
since $P_1, P_4, \ldots, P_{s}, P_{s+1} P_{s+2}, P_{s+3}$ do not lie on a linear
$(s-1)$-space, there exists a hyperplane, say $\sigma_2$, containing $P_1, P_4,
\ldots, P_{s}, P_{s+1}, P_{s+2}$ and avoiding $P_{s+3}$. Moreover, since $P_1,
\ldots, P_{s+1} \in H$ and $P_1, P_2, P_4, \ldots, P_s, P_{s+2}\in L$, we have
$$H^{m-i-1}L^{m-2}\sigma_2\sigma_1 \in \wp_1^{2m-i-1} \cap \wp_2^{m-i}\cap \wp_3^{m-i}\cap
\wp_4^{2m-i-1}\cap \cdots \cap \wp_s^{2m-i-1}\cap \wp_{s+1}^{m-i} \cap
\wp_{s+2}^{m}.$$ Since $2m-i-1 \ge m$, we have
$$H^{m-i-1}L^{m-2}\sigma_2\sigma_1 \in \wp_1^{m} \cap \wp_2^{m-i}\cap \wp_3^{m-i}\cap
\wp_4^{m}\cap \cdots \cap \wp_s^{m}\cap \wp_{s+1}^{m-i} \cap \wp_{s+2}^{m}.$$
Moreover, since $M\in \wp_2^{i} \cap \wp_3^{i}\cap \wp_5^{i} \cap\cdots \cap
\wp_{s}^{i} \cap  \wp_{s+1}^{i}$, we have
\begin{align} H^{m-i-1}L^{m-2}\sigma_2\sigma_1 M \in \wp_1^m \cap \cdots \cap
\wp_{s+2}^m=J \end{align} with $m-i-1+m-2+2 +i = 2m-1=T_1$.

\par\smallskip $\bullet$  $c_s = i$ and $m=2$ and $s\ge 5$. Since $m=2$ and $1\le i \le m-1$, we get $i=1$.
Let $\beta_1$ be a linear $(s-1)$-space containing $s$ points $P_1, P_3, \ldots,
P_{s-1}, P_{s+1}, P_{s+2}$. If $P_{s+3}\in \beta_1$, then $s$ points $P_3,
\ldots, P_{s-1}, P_{s+1}, P_{s+2}$ are contained in the linear $(s-2)$-space
$\beta_1 \cap \beta$. This contradicts our assumption. So, $P_3 \notin \beta_1$.
 Therefore, there exists a hyperplane, say $\varrho$, containing $\beta_1$ and avoiding $P_{s+3}$. Recall that $L$
contains $P_1, P_2, P_4, \ldots, P_s, P_{s+2}$. Thus, we have
$$\varrho L \in \wp_1^2 \cap \wp_2 \cap \wp_3 \cap \wp_4^2 \cap \wp_5 \cap \cdots
\cap \wp_{s+1} \cap \wp_{s+2}^2.$$ Moreover, since $M\in \wp_2 \cap \wp_3 \cap
\wp_5 \cap\cdots \cap \wp_{s} \cap  \wp_{s+1}$, we have \begin{align} \varrho L
M \in \wp_1^2 \cap \cdots \cap \wp_{s+2}^2=J \end{align} with $1+1+i =3=T_1$.

From $(1), (2), (3), (4)$ and by Remark \ref{remark23} we get
$$\reg(R/(J+\wp^m_{s+3}))\le T_1.$$

\end{proof}

\par\smallskip

We need the following proposition to find a sharp upper bound for the regularity
index of $s+3$ equimultiple fat points not on a linear $(s-1)$-space.\par

\par\smallskip

\begin{proposition}\label{prop43} Let $X=\{P_1, \ldots, P_{s+3}\}$ be a
set of distinct points lie on a linear $s$-space $\gamma$ but $X$ is not in
general position on $\gamma$ and $X$ does not lie on a linear $(s-1)$-space in
$\mathbb P^n$, $2\le s \le n$. Let  $m$ be a positive integer. Assume that
$\wp_1, \ldots, \wp_{s+3}$ are the homogeneous prime ideals of the polynomial
ring $R=K[X_0, \ldots, X_n]$ corresponding to the points $P_1, \ldots, P_{s+3}$.
For $j=1, \ldots, n$, put
$$T_j = \max\left\{\left[\frac{mq+ j- 2}{j}\right] |\
P_{i_1}, \ldots , P_{i_q} \text{ lie on a linear }j\text{-space}\right\}.$$
Then, there exists a point $P_{i_0}\in X$ such that
$$\reg(R/(J+\wp_{i_0}^m))\le \max\{T_j|j=1,\ldots,n\},$$
where
$$J=\underset{i\ne i_0}{\cap}\wp_i^m.$$
\end{proposition}
\begin{proof} We have $T_1 \ge 2m-1$ and $T_s=\left[\frac{(s+3)m+ s- 2}{s}\right] \ge T_{s+1} \ge \cdots \ge T_n$. We consider two following cases:

\par\smallskip\noindent {\it\bf Case 1:} There exists a hyperplane, say $H$, avoiding a point
of $X$ and passing through $s+2$ points of $X$. We may assume that $P_{s+3}
\notin H$ and $P_1, \ldots, P_{s+2} \in H$ (after relabeling, if necessary). Put
$P_{i_0}=P_{s+3}=(1, 0, \ldots, 0)$. For every monomial $M=X_1^{c_1}\cdots
X_{n}^{c_n}$, $c_1+\cdots +c_n=i$, $i=0, \ldots, m-1$. We have $H^m \in J$. It
implies that $H^m M \in J$. By Remark \ref{remark23} we get
$$\reg(R/(J+\wp^m_{i_0}))\le \max\{m+i| i=0, \ldots, m-1\}\le 2m-1 \le T_1.$$

\par\smallskip\noindent {\it\bf Case 2:} There does not exist any hyperplane avoiding a point
of $X$ and passing through $s+2$ points of $X$. This implies that there does not
exist any linear $(s-1)$-space passing through $s+2$ points of $X$. So, a linear
$(s-1)$-space contains at most $s+1$ points of $X$. Since $X$ is not in general
position in the linear $s$-spcace $\gamma$, there exists a linear $(s-1)$-space
containing $s+1$ points of $X$. Put
$$k:=\min\{h| \text{ there exists a linear } h \text{-space containng } h+2 \text{ points of }
X\}.$$ Then, $k\le s-1$. Since a linear $(s-1)$-space containing at most $s+1$
points of $X$, we have a linear $h$-space containing at mots $h+2$ points of $X$
with $h\le s-1$. Thus, $T_k \ge T_h$, $h=k+1, \ldots, s-1$.

\par Let $\alpha$ be a linear $k$-space containing $k+2$ points of $X$. We may
assume that $P_1, \ldots, P_{k+2} \in \alpha$ and $P_{k+3}, \ldots, P_{s+3}
\notin \alpha$ (after relabeling, if necessary). We consider the two following
subcases:

\par\smallskip\noindent {\it\bf Case 2.1:} $m=1$ or $s-k \ge k$ or $s-k\ge 2$.
For arbitrary $s$ points of $X$, there always exists a linear $(s-1)$-space
containing them. Thus, there is a linear $(s-1)$-space, say $\beta$, such that
$\beta$ contains $P_{k+3}, \ldots, P_{s+3}$ and $\beta$ contains $k-1$ points of
$X\cap \alpha$. We are considering Case 2, so $\beta$ avoids two points of $X
\cap \alpha$. We may assume that $P_4, \ldots, P_{k+2} \in \beta$, $P_1\notin
\beta$, $P_2 \notin \beta$.\par

By the property of $k$, we have arbitrary $k+1$ points of $X\cap \alpha$ not on
a linear $(k-1)$-space. We can put $P_{i_0}=P_1=(1, 0, \ldots, 0)$, $P_2=(0,
\underset{2}{\underbrace{1}}, 0, \ldots, 0)$, $P_3=(0, 0,
\underset{3}{\underbrace{1}}, 0, \ldots, 0)$, $P_5=(0, 0, 0,
\underset{4}{\underbrace{1}}, 0, \ldots, 0)$,..., $P_j=(0, \ldots, 0,
\underset{j-1}{\underbrace{1}}, 0, \ldots, 0)$, $j=5, \ldots, k+2$. Also by the
property of $k$, we have $P_1, P_2, P_3, P_5,  \ldots, P_{k+2}$ and arbitrary
$s-k-1$ points of $X\setminus \alpha$  not on a linear $(s-2)$-space. We can put
$P_j=(0, \ldots, 0, \underset{j-1}{\underbrace{1}}, 0, \ldots, 0)$, $j=k+3,
\ldots, s+2$.

\par For every monomial $M=X_1^{c_1}\cdots X_{n}^{c_n}$, $c_1+\cdots +c_n=i$,
$i=0, \ldots, m-1$. Put  $m_4=m_{s+3}=m$, $m_2=m-i+c_1$, $m_3=m-i+c_2$,
$m_j=m-i+c_{j-2}$, $j=5, \ldots, s+2$. Put $$t_1=\max\left\{m_j,
\left[\frac{\sum_{l=2}^{k+2} m_{l}+ k- 1}{k}\right]|\ j=2, \ldots k+2
\right\}.$$ Since there always exists a linear $(k-1)$-space passing arbitrary
$k$ points of $\{P_2, \ldots, P_{k+2}\}$ and avoiding $P_{i_0}$, by Lemma
\ref{lem41} we can find $t_1$ linear $(k-1)$-space avoiding $P_{i_0}$, say $L_1,
\ldots, L_{t_1}$, such that  for every point $P_j\in \{P_2, \ldots, P_{k+2}\}$,
there are $m_j$ linear $(k-1)$-spaces (including multiplicity) of $\{L_1,
\ldots, L_{t_1}\}$ passing through the $P_j$.

Consider the set $\{ P_{k+3}, \ldots, P_{s+3}\}$. We remark that there is not
any linear $(s-k-1)$-space containing $P_{i_0}$ and $s-k$ points of $\{ P_{k+3},
\ldots, P_{s+3}\}$. In fact, assume that there exists a linear $(s-k-1)$-space,
say $L'$, containing $P_{i_0}$ and $s-k$ points of $\{ P_{k+3}, \ldots,
P_{s+3}\}$. Then, linear $(s-1)$-space, say $H$, containing the linear
$(k-1)$-space $L_1$ and the linear $(s-k-1)$-space $L'$ contains $P_{i_0}$.
Therefore, this linear $(s-1)$-space $H$ contains $\alpha$. So, $H$ contains
$s+2$ points of $X$, a contradiction.

Put
$$t_2=\max\left\{m_j, \left[\frac{\sum_{l=k+3}^{s+3} m_{l}+ s-k-1}{s-k}\right]|\ j=k+3,
\ldots s+3 \right\}.$$ Since there always exists a linear $(s-k-1)$-space
passing arbitrary $s-k$ points of $\{P_{k+3}, \ldots, P_{s+3}\}$, by the above
remark we can find $t_2$ linear $(s-k-1)$-space avoiding $P_{i_0}$, say $L'_1,
\ldots, L'_{t_2}$, such that for every point $P_j\in \{P_{k+2}, \ldots,
P_{s+3}\}$, there are $m_j$ linear $(s-k-1)$-spaces (including multiplicity) of
$\{L'_1, \ldots, L'_{t_2}\}$ passing through the $P_j$. Put $$t=\max\{t_1,
t_2\}.$$ For $j=1, \ldots, t$, there always exists a linear $(s-1)$-space, say
$H_j$, containing $L_j$, $L'_j$ and avoiding $P_{i_0}$ (If $H_j$ contains
$P_{i_0}$, then $H_j$ contains $\alpha$ and $L'_j$. This implies that $H_j$
contains $s+2$ points of $X$, a contradiction). Since $H_j$ contains $L_j$ and
$L'_j$, $j=1, \ldots, t$, we get that for every point $P_i \in \{ P_2, \ldots,
P_{s+3}\}$, there are $m_i$ hyperplanes of $\{ H_1, \ldots, H_t\}$ passing
through the $P_i$. Therefore,
$$H_1\cdots H_t \in \wp_2^{m_2} \cap \cdots \cap \wp_{s+3}^{m_{s+3}}=\wp_2^{m-i+c_1} \cap \wp_3^{m-i+c_2}\cap
\wp_4^m \cap \wp_5^{m-i+c_3} \cap \cdots \cap \wp_{s+2}^{m-i+c_s}\cap
\wp_{s+3}^m.$$ Moreover, since $M \in \wp_2^{i-c_1}\cap \wp_3^{i-c_2} \cap
\wp_5^{i-c_3}\cap \cdots \cap \wp_{s+2}^{i-c_s}$, we get
$$H_1\cdots H_t M \in \wp_2^m \cap \cdots \cap \wp_{s+3}^m=J.$$
By Remark \ref{remark23} we get
$$\reg(R/(J+\wp^m_{i_0}))\le \max\{t+i| i=0, \ldots, m-1\}.$$
We recall that
$$t=\max\{t_1, t_2\}=\max\left\{m, \left[\frac{\sum_{l=2}^{k+2} m_{l}+ k-
1}{k}\right], \left[\frac{\sum_{l=k+3}^{s+3} m_{l}+
s-k-1}{s-k}\right]\right\}.$$

\par\noindent{\it\bf Case 2.1.1:} If $t=m$, then $$\max\{t+i| i=0, \ldots, m-1\}\le 2m-1 \le T_1.$$

\par\noindent{\it\bf Case 2.1.2:} If $t=\left[\frac{\sum_{l=2}^{k+2} m_{l}+ k- 1}{k}\right]$, then
\begin{align*}
t+i&=\left[\frac{\sum_{l=2}^{k+2} m_{l}+ k- 1+ki}{k}\right]=
\left[\frac{(k+1)m-ki+\sum_{l=3}^{k+2} c_{l}+ k- 1+ki}{k}\right]\\
&\le \left[\frac{(k+1)m +i+ k- 1}{k}\right]\le \left[\frac{(k+2)m +
k-2}{k}\right]=T_k.
\end{align*}
So, $$\max\{t+i| i=0, \ldots, m-1\} \le T_k.$$

\par\noindent{\it\bf Case 2.1.3:} If $t=\left[\frac{\sum_{l=k+3}^{s+3} m_{l}+ s-k-1}{s-k}\right]$, then
\begin{align*}
t+i&=\left[\frac{\sum_{l=k+3}^{s+3} m_{l}+ s-k-1+(s-k)i}{s-k}\right]\\
&=\left[\frac{(s-k+1)m-(s-k)i+\sum_{l=k+3}^{s+2}
c_{l-2}+s-k-1+(s-k)i}{s-k}\right]\\
&= \left[\frac{(s-k+1)m+\sum_{l=k+3}^{s+2} c_{l-2} +s-k-1}{s-k}\right].
\end{align*}
Note that $\sum_{l=k+3}^{s+2} c_{l-2} \le i \le m-1$. Recall that we are
considering Case 2.1: $m=1$ or $s-k \ge k$ or $s-k \ge 2$. \par

\par\noindent{\it\bf Case 2.1.3.1:} If $m=1$, then  we have
\begin{align*}\max\{t+i| i=0, \ldots, m-1\}
&=\left[\frac{(s-k+1)1+s-k-1}{s-k}\right] =2\\
&\le \left[\frac{(s+3)1+s-2}{s}\right]=T_s.\end{align*}

\par\noindent{\it\bf Case 2.1.3.2:} If $s-k \ge k$, then we have
\begin{align*}\max\{t+i| i=0, \ldots, m-1\}
&= \left[\frac{(s-k+1)m+\sum_{l=k+3}^{s+2} c_{l-2} +s-k-1}{s-k}\right]\\
&\le \left[\frac{(s-k+1)m+m-1 +s-k-1}{s-k}\right]\\
&\le\left[\frac{(s-k+2)m+s-k-2}{s-k}\right]\\
&\le \left[\frac{(k+2)m+k-2}{k}\right]=T_k. \end{align*} The next step we need
only consider case of $s-k \ge 2$ and $m\ge 2$ and $s-k <k$.

\par\noindent{\it\bf Case 2.1.3.3:} If $s-k \ge 3$ and $m\ge 2$, then we have
\begin{align*}\max\{t+i| i=0, \ldots, m-1\} &=\left[\frac{(s-k+1)m+\sum_{l=k+3}^{s+2} c_{l-2} +s-k-1}{s-k}\right]\\
&\le\left[\frac{(s-k+2)m+s-k-2}{s-k}\right]\le 2m-1 \le T_1.\end{align*}

\par\noindent{\it\bf Case 2.1.3.4:} If $s-k=2$ and  $s-k < k$ and $m\ge 2$, then since $c_1+\ldots +c_s
=i \le m-1$, we have $\sum_{l=k+3}^{s+2} c_{l-2} \le m-1$.  We consider two
following for $\sum_{l=k+3}^{s+2} c_{l-2}$:

$\bullet$ If $\sum_{l=k+3}^{s+2} c_{l-2} \le m-2$, then we have
$$H_1\cdots H_t M \in \wp_2^m \cap \cdots \cap \wp_{s+3}^m=J$$
with $$\max\{t+i| i=0, \ldots, m-1\}= \left[\frac{3m+\sum_{l=k+3}^{s+2}
c_{l-2}+1}{2}\right] \le \left[\frac{4m-1}{2}\right] = 2m-1 \le T_1.$$

$\bullet$  If $\sum_{l=k+3}^{s+2} c_{l-2}= m-1$, then we have $c_j=0, j=1,
\ldots, s-2$, $i=m-1$. Therefore, $m_2=m_3=m_5=\cdots=m_{s}=1$ and $$M\in
\wp_2^{m-1} \cap \wp_3^{m-1} \cap \wp_5^{m-1}\cap \cdots \cap \wp_s^{m-1}\cap
\wp_{s+1}^{m-1-c_{s-1}}\cap \wp_{s+2}^{m-1-c_{s}}.$$ We recall that the $\beta$
in Case 2.1 is the linear $(s-1)$-space containing $P_4, \ldots, P_{s+3}$ and
avoiding $P_{i_0}$. Let $K_1$ be the hyperplane containing $\beta$ and avoiding
$P_{i_0}$. Then, we have
$$K_1M \in \wp_2^{m-1}\cap \wp_3^{m-1} \cap \wp_4 \cap \wp_5^m
\cdots \cap \wp_s^m \cap \wp_{s+1}^{m-c_{s-1}}\cap \wp_{s+2}^{m-c_{s}}\cap
\wp_{s+3}.$$ The linear $(s-1)$-space, say $\gamma_1$, containing $P_2, P_3,
P_4, \ldots, P_{s-1}, P_{s+1}, P_{s+3}$ avoids $P_{i_0}$ (If $P_{i_0}\in
\gamma_1$, then $\gamma_1$ contains $\alpha$. So, $\gamma_1$ is a linear
$(s-1)$-space containing $s+2$ points of $X$, a contradiction). Similarly, the
linear $(s-1)$-space, say $\gamma_2$, containing $P_2, P_3, P_4, \ldots,
P_{s-1}, P_{s+2}, P_{s+3}$ avoids $P_{i_0}$. Let $K_2$ be a hyperplane
containing $\gamma_1$ and avoiding $P_{i_0}$. Let $K_3$ be a hyperplane
containing $\gamma_2$ and avoiding $P_{i_0}$. Then, we have
$$K_3^{c_{s}}K_2^{c_{s-1}}K_1M \in \wp_2^{m}\cap \wp_3^{m} \cap \wp_4^{1+c_{s-1}+c_{s}} \cap
\wp_5^m \cap  \cdots \cap \wp_s^m \cap \wp_{s+1}^{m}\cap \wp_{s+2}^{m}\cap
\wp_{s+3}^{1+c_{s-1}+c_s}=\wp_1^m \cap \cdots \cap \wp_{s+3}^m=J$$ with
$$c_s+c_{s-1}+1+i=2m-1.$$ So, in case of $s-k=2$, $s-k < k$ and $m\ge 2$ by
Remark \ref{remark23} we get
$$\reg(R/(J+\wp^m_{i_0}))\le  2m-1= T_1.$$

\par\smallskip\noindent {\it\bf Case 2.2:}  $m\ge 2$ and $s-k=1$ and $s-k<k$. Then $\alpha$ contains
$P_1,\ldots, P_{s+1}$ and there is not any linear $(s-2)$-space containing $s$
points of $X$. We consider two following cases:

\par\smallskip\noindent {\it\bf Case 2.2.1:} There is not any linear $(s-1)$-space
containing $P_{s+2}$, $P_{s+3}$ and $s-1$ points of $X\cap \alpha$. So, in this
case every linear $(s-1)$-space containing arbitrary $s$ points of $\{P_1,
\ldots, P_{s+2}\}$ avoids $P_{s+3}$. Put $P_{i_0}=P_{s+3}=(1, 0, \ldots, 0)$,
$P_1=(0, \underset{2}{\underbrace{1}}, 0, \ldots, 0)$, ..., $P_j=(0, \ldots, 0,
\underset{j+1}{\underbrace{1}}, 0, \ldots, 0)$, $j=1, \ldots, s$. For every
monomial $M=X_1^{c_1}\cdots X_{n}^{c_n}$, $c_1+\cdots +c_n=i$, $i=0, \ldots,
m-1$. Put  $m_j=m-i+c_{j}$, $j=1, \ldots, s$, $m_{s+1}=m_{s+2}=m$. Put
$$t=\max\left\{m_j, \left[\frac{\sum_{l=1}^{s+2} m_{l}+ s- 1}{s}\right]|\ j=1,
\ldots s+2 \right\}.$$ By Lemma \ref{lem41} we can find $t$ linear $(s-1)$-space
avoiding $P_{i_0}$, say $L_1, \ldots, L_{t}$, such that  for every point $P_j\in
\{P_1, \ldots, P_{s+2}\}$, there are $m_j$ linear $(s-1)$-spaces (including
multiplicity) of $\{L_1, \ldots, L_{t}\}$ passing through the $P_j$. \par

For $j=1, \ldots, t$, let $H_j$ be a hyperplane containing $L_j$ and avoiding
$P_{i_0}$. Then we have
$$H_1 \cdots H_t \in \wp_1^{m_1}\cap  \cdots \cap \wp_{s+2}^{m_{s+2}}=\wp_1^{m-i+c_1}\cap \cdots \cap \wp_s^{m-i+c_s}
\cap \wp_{s+1}^m \cap  \wp_{s+2}^m.$$ Moreover, since $M\in \wp_1^{i-c_1} \cap
\cdots \cap \wp_{s}^{i-c_{s}}$, we have
$$H_1 \cdots H_t M \in \wp_1^m \cap \cdots \cap \wp_{s+2}^m=J.$$
By Remark \ref{remark23} we get
$$\reg(R/(J+\wp^m_{i_0}))\le \max\{ t+i| i=1, \ldots, m-1\}.$$
If $t=m$, then
$$\max\{ t+i| i=1, \ldots, m-1\} \le 2m-1 =T_1.$$
If $t=\left[\frac{\sum_{l=1}^{s+2} m_{l}+ s- 1}{s}\right]$, then
$$\max\{ t+i| i=1, \ldots, m-1\}\le \left[\frac{\sum_{l=1}^{s+3} m_{l}+ s- 2}{s}\right]=T_s.$$

\par\smallskip\noindent {\it\bf Case 2.2.2:} There is a linear $(s-1)$-space, say $\beta$,
containing $P_{s+2}$, $P_{s+3}$ and $s-1$ points of $X\cap \alpha$. We may
assume that $P_3, \ldots, P_{s+2}\in \beta$. Put $P_{i_0}=P_{s+3}$ and
$J=\wp_1^m \cap \cdots \cap \wp_{s+2}^m$. If $s=2$, then $P_3, P_4, P_5$ lie on
the line $\beta$ and $P_1, P_2 \notin \beta$. Let $Q_1$ be the linear
$(n-1)$-space passing throught $P_3, P_1$ and avoiding $P_5$. Let $Q_2$ be the
linear $(n-1)$-space passing throught $P_4, P_2$ and avoiding $P_5$. Then
$$Q_1^m Q_2^m M\in J$$
for every monomial $M=X_1^{c_1}\cdots X_{n}^{c_n}$, $c_1+\cdots +c_n=i$, $i=0,
\ldots, m-1$. By Remark \ref{remark23} we get
$$\reg(R/(J+\wp^m_{i_0}))\le \max\{2m +i |i=1, \ldots, m-1\} \le 3m-1= T_1.$$
If $s\ge 3$, by Lemma \ref{lem42} we get
$$\reg(R/(J+\wp_{i_0}^m))\le \max\{T_j|j=1,\ldots,n\}.$$

\par The proof of Proposition \ref{prop43} is completed.

\end{proof}

The following proposition gives a sharp upper bound for the regularity index of
$s+3$ equimultiple fat points not on a linear $(s-1)$-space.

\par\smallskip \begin{theorem}\label{thm44} Let $X=\{P_1, \ldots, P_{s+3}\}$
be a set of distinct points not on a linear $(s-1)$-space in $\mathbb P^n$, $s
\le n$, and $m$ be a positive integer. Let
$$Z=mP_1+\cdots+mP_{s+3}$$
be the equimultiple fat points. Then,
$$\reg(Z)\le \max\{T_j|\ j=1,\ldots,n\},$$
where
$$T_j = \max\left\{\left[\frac{mq+ j- 2}{j}\right] |\
P_{i_1}, \ldots , P_{i_q} \text{ lie on a linear }j\text{-space}\right\}.$$
\end{theorem}
\begin{proof} Let
$\wp_1, \ldots, \wp_{s+3}$ be the homogeneous prime ideals of the polynomial
ring $R=K[X_0, \ldots, X_n]$ corresponding to the points $P_1, \ldots, P_{s+3}$.
Put $I=\wp_1^m \cap \cdots \cap \wp_{s+3}^m$. We have $\reg(Z)=\reg(R/I)$.

We argue by induction on $s$. For $s=1$, the theorem is true  by Lemma
\ref{lem28}. We assume that the theorem is true for $s-1$. By Proposition
\ref{prop43}, there exists a points $P_{i_0} \in X$ such that
\begin{align}\reg(R/(J+\wp_{i_0}^m))\le \max\{T_j|j=1,\ldots,n\}, \end{align} where
$$J=\underset{i\ne i_0}{\cap}\wp_i^m.$$
Put $Y=X\setminus \{P_{i_0}\}$. Since $X$ does not lie on a linear
$(s-1)$-space, we have $Y$ does not lie on a linear $(s-2)$-space. Put
$$T'_j=\max\left\{ \left[ \frac{m_{H\setminus \{P_{i_0}\}}+j-2}{j} | H \text{ is a linear }j\text{-space}   \right]\right\}$$
with $$m_{H \setminus \{ P_{i_0}\}} = \sum_{P_i \in H \setminus \{ P_{i_0}\}}
m_i,\  m_i=m.$$ By inductive assumption, we get
$$\reg(R/J) \le \max\left\{T'_j | j+1, \ldots, n \right\}.$$
We have $T'_j \le T_j$, $j=1, \ldots, n$. Thus, \begin{align}\reg(R/J) \le
\max\{T_j|j=1,\ldots,n\}. \end{align} By Lemma \ref{lem21} we have
\begin{align}\reg(R/I)=\max\left\{m-1, \reg(R/J),
\reg(R/(J+\wp_{i_0}^m))\right\}.\end{align} Therefore, from $(5)$, $(6)$ and
$(7)$ we get
$$\reg(R/I)\le \max\{T_j|j=1,\ldots,n\}.$$

\par The proof of Theorem \ref{thm44} is completed.
\end{proof}

\bigskip

\noindent Phan Van Thien \\ Department of
Mathematics, Hue Normal University, Vietnam \\
Email: tphanvannl@yahoo.com\\

\bigskip

\noindent Ho Thi Doan Trang \\ Department of
Mathematics, Hue Normal University, Vietnam\\
Email: hothidoantrang82@yahoo.com

\end{document}